\documentclass[11pt]{article}

\usepackage[margin=1in]{geometry}
\usepackage{amsmath, amssymb, amsthm, mathtools}
\usepackage{microtype}
\usepackage{enumitem}
\usepackage{pgfplots}
\pgfplotsset{compat=1.18}

\newtheorem{theorem}{Theorem}
\newtheorem{proposition}{Proposition}
\newtheorem{lemma}{Lemma}
\newtheorem{remark}{Remark}
\newtheorem{corollary}{Corollary}

\newcommand{\R}{\mathbb{R}}
\newcommand{\N}{\mathbb{N}}
\newcommand{\argmin}{\operatorname*{argmin}}
\newcommand{\prox}{\operatorname{prox}}

\newcommand{\gradmap}{\mathcal{G}}

\newcommand{\abs}[1]{\lvert{#1}\rvert}

\usepackage[dvipsnames]{xcolor}

\usepackage[
    colorlinks=true,
    linkcolor=MidnightBlue,
    citecolor=MidnightBlue,
    urlcolor=RoyalBlue
]{hyperref}

\usepackage[authoryear]{natbib}
\usepackage{authblk}

\title{AdaGrad does not adapt to H\"older-smoothness \\ for composite objectives}

\author{
Matia Bojovi\'c$^{1,2}$
\quad
Saverio Salzo$^{1,3}$
\quad
Massimiliano Pontil$^{1,4}$
\\
{\small $^1$ Computational Statistics and Machine Learning, Istituto Italiano di Tecnologia, Genoa, Italy}
\\
{\small $^2$ Department of Mathematics, University of Genoa, Genoa, Italy}
\\
{\small $^3$ DIAG, Sapienza University of Rome, Rome, Italy}
\\
{\small $^4$ Department of Computer Science, University College London, London, UK}
}

\date{}
\usepackage[textsize=tiny]{todonotes}

\begin{document}

\maketitle
\begin{abstract}
    \noindent 
    We exhibit a simple deterministic one-dimensional convex composite optimization problem for which AdaGrad scheme does not achieve the classical convergence rate $\mathcal{O}(n^{-(1+\nu)/2})$ associated with  H\"older-smooth objectives. The example highlights a basic mismatch between classical AdaGrad accumulation and composite optimality. A main insight is that the gradient of the smooth term may not vanish at the optimum, causing AdaGrad to keep reducing its stepsize excessively and converge more slowly.  
    We also discuss why alternative accumulation mechanisms based on gradient mappings or on successive gradient differences, avoid this pathology.
\end{abstract}

\section{Introduction}

Adaptive gradient methods are among the standard tools for training machine learning models. Their appeal is that they reduce the need to tune a fixed learning rate by adjusting the effective stepsize using information observed along the optimization trajectory. 
AdaGrad, introduced by \citet{JMLR:v12:duchi11a}, is a prototypical example: it rescales the update by the square root of the cumulative sum of past squared subgradients, coordinate by coordinate. The method was originally proposed for nonsmooth Lipschitz-continuous composite convex optimization, achieving the optimal rate \(\mathcal{O}(1/\sqrt n)\) in the objective gap. \\

\noindent
Later works considered the smooth setting and asked whether AdaGrad can adapt to the unknown smoothness level of the objective, while attaining the corresponding standard rate. This question is closely related to the notion of universality introduced by \citet{nesterov2015universal}, namely the ability of an algorithm to achieve the optimal rate without prior knowledge of the smoothness level.
For AdaGrad, an early result in this direction was given by \citet{levy2018online}, who, to the best of our knowledge, is the first to show that AdaGrad
also adapts to smooth problems with Lipschitz-continuous gradients, achieving the standard rate \(\mathcal O(1/n)\). The settings described above are the endpoints of a broader class of functions, which is that of H\"older-smoothness \(C^{1,\nu}\), $\nu\in[0,1]$. For this class of functions, \citet{orabona2023normalized} has proved that AdaGrad achieves the rate \(\mathcal O(n^{-(1+\nu)/2})\), fully interpolating between \(\mathcal{O}(1/ \sqrt n)\) and \(\mathcal{O}(1/n)\). \\

\noindent
However, all these results are formulated in the noncomposite setting, where the update reduces to a gradient step with an adaptive stepsize or diagonal scaling.
In this simplified setting, if \(x_\star\) minimizes a differentiable convex function \(f\) over the whole space, then
\[
\nabla f(x_\star)=0.
\]
Thus, along a convergent deterministic trajectory, the quantities accumulated by AdaGrad, namely the squared gradient norms, are expected to vanish. The growth of the accumulator is therefore controlled by progress toward stationarity: as the iterates approach a minimizer, AdaGrad keeps adding smaller and smaller terms, allowing for the appropriate convergence rates in the given class of functions.  \\

\noindent
Many problems in machine learning and statistics are naturally formulated in composite form,
\begin{equation}
    \min_{x\in\R^d} F(x) := f(x)+\varphi(x),
\label{eq:composite-problem}
\end{equation}
where \(f\colon\R^d\to\R\) is convex and smooth, while \(\varphi\colon\R^d\to\R\cup\{+\infty\}\) is proper, convex and lower semicontinuous. In applications, \(\varphi\) often represents nonsmooth regularization, such as an \(\ell_1\)-penalty, or constraints through an indicator function. The point we emphasize, however, is not nonsmoothness itself, but the splitting of the objective. In a composite problem, stationarity is not measured by the gradient of the smooth part alone. The first-order optimality condition is
\[
0\in \nabla f(x_\star)+\partial\varphi(x_\star),
\]
where \(\partial\varphi(x_\star)\) denotes the subdifferential of \(\varphi\) at \(x_\star\). Thus, the gradient of $f$ does not vanish at a minimizer, since it may be balanced by the first-order contribution of \(\varphi\). This creates a mismatch for AdaGrad when its metric is driven only by \(\lVert\nabla f(x_n)\rVert\): even if the iterates approach a composite minimizer, the accumulated gradient norms may keep growing linearly. Consequently, the effective stepsizes may decay as if the method were not approaching stationarity. \\

\noindent
The purpose of this note is to make this obstruction explicit. We construct a deterministic one-dimensional convex composite problem for which the smooth part is H\"older-smooth, but the gradient of the smooth part remains bounded away from zero along the AdaGrad trajectory. The example uses an affine \(\varphi\), showing that the obstacle is not caused by nonsmoothness of the regularizer, but by gradient's accumulation applied only to one component of the objective. As a consequence, the AdaGrad weights grow at least as the square root of the iteration counter, and the resulting effective stepsize becomes too small to retain the appropriate H\"older-smooth composite rate. \\ [-1ex]

\noindent
\textbf{Contributions. }Our main result shows that for every \(\nu\in\left]0,1\right]\) and $\alpha\in\left]1/{2},(1+\nu)/{2}\right[$, there exists a one-dimensional convex composite objective \(F=f+\varphi\), with \(f\in C^{1,\nu}(\R)\), $\varphi\in\Gamma_0(\R)$, such that for every \(\varepsilon\ge0\) and every base stepsize \(\eta\) in a suitable range around zero, the corresponding AdaGrad iterates 
\begin{equation}\label{update}
    x_{n+1}
    =
    \prox_{\frac{\eta}{w_n}\varphi}
    \left(
        x_n-\frac{\eta}{w_n}f'(x_n)
    \right), \quad {\rm where}\,\, w_n
    =
    \varepsilon+
    \left(
        \sum_{k=0}^n |f'(x_k)|^2
    \right)^{1/2}
\end{equation}
satisfy
\[
    F(x_n)-F_\star
    \ge \kappa
    n^{-\alpha},
\]
for every $n\in\N$. This rules out the expected rate  \(\mathcal{O}(n^{ -(1+\nu)/2})\)
for AdaGrad in the deterministic composite setting, marking a conceptual difference with the unconstrained scenario. \\

\noindent
\textbf{Related works. }The issue of the accumulation mechanism of the gradients in AdaGrad has already been considered in the literature. Indeed, there exist works proposing different accumulation schemes. 
\citet{antonakopoulos2025adaptive,wang2026universal} analyze an AdaGrad-type algorithm, which accumulates gradient mapping evaluations instead of gradients. Indeed, in a composite setting, the gradient mapping fully characterizes the stationarity and can effectively serve as an accumulation scheme. On the other hand, in \citep{bojovic2026adagrad}, the accumulation is driven by successive gradient differences. In this case, a nonzero limiting value of $\nabla f(x_n)$ does not by itself force linear growth of the metric. Additional considerations are given in Section~\ref{section:4}. \\

\noindent
\textbf{Notation. }We denote by \(\Gamma_0(\mathcal \R)\) the class of functions
\(h\colon\mathcal \R\to\R\cup\{+\infty\}\) that are proper, convex, and lower
semicontinuous. If \(h\in\Gamma_0(\mathcal \R)\), we denote by
\(\argmin h\) the set of its minimizers, while its subdifferential is the set-valued mapping
$\partial h\colon \R\to 2^\R$
such that, for every \(x\in\mathcal \R\), by
\(
    \partial h(x)
    :=
    \left\{
        u\in\mathcal \R
        \ \middle|\
        h(y)\ge h(x)+\langle{y-x},\;{u}\rangle
        \text{ for every }y\in\mathcal \R
    \right\}.
\)
For $x\in\R$, $\alpha>0$, we define $\prox_{\alpha h}(x)
    :=
    \argmin_{y\in\R}
    \left\{
        h(y)+\alpha\vert y-x\vert^2
    \right\}$. For $\nu\in \left]0,1\right]$ and $h\colon\R\to\R$, we say that \(h\in C^{1,\nu}(\R)\) if \(h\) is differentiable and there exists \(L_\nu>0\) such that $|h'(x)-h'(y)|
    \le
    L_\nu |x-y|^\nu
    \;
    \text{for every }x,y\in\R$.

\begin{figure}[t]
\centering
\begin{tikzpicture}

\def\cval{0.8}

\begin{axis}[
    name=plotf,
    width=0.47\textwidth,
    height=0.34\textwidth,
    axis lines=middle,
    xlabel={$x$},
    ylabel={$f(x)$},
    xmin=-2.1, xmax=2.1,
    ymin=-1.35, ymax=2.45,
    samples=500,
    domain=-2.1:2.1,
    grid=major,
    legend style={
        draw=none,
        fill=none,
        at={(1.10,-0.18)},
        anchor=north,
        legend columns=3
    },
]

\addplot[thick, blue]
{
    (abs(x) <= 1 ? \cval*(abs(x)^4)/4 : \cval*(abs(x)-1+1/4)) - x
};
\addlegendentry{$p=4$}

\addplot[thick, red]
{
    (abs(x) <= 1 ? \cval*(abs(x)^6)/6 : \cval*(abs(x)-1+1/6)) - x
};
\addlegendentry{$p=6$}

\addplot[thick, teal]
{
    (abs(x) <= 1 ? \cval*(abs(x)^8)/8 : \cval*(abs(x)-1+1/8)) - x
};
\addlegendentry{$p=8$}

\addplot[dashed, gray, forget plot] coordinates {(-1,-1.35) (-1,2.45)};
\addplot[dashed, gray, forget plot] coordinates {(1,-1.35) (1,2.45)};

\end{axis}

\begin{axis}[
    at={(plotf.east)},
    xshift=0.07\textwidth,
    anchor=west,
    width=0.47\textwidth,
    height=0.34\textwidth,
    axis lines=middle,
    xlabel={$x$},
    ylabel={$f'(x)$},
    xmin=-2.1, xmax=2.1,
    ymin=-1.95, ymax=0.5,
    samples=500,
    domain=-2.1:2.1,
    grid=major,
]

\addplot[thick, blue]
{
    (abs(x) <= 1 ? \cval*(x >= 0 ? abs(x)^3 : -abs(x)^3) : \cval*(x >= 0 ? 1 : -1)) - 1
};

\addplot[thick, red]
{
    (abs(x) <= 1 ? \cval*(x >= 0 ? abs(x)^5 : -abs(x)^5) : \cval*(x >= 0 ? 1 : -1)) - 1
};

\addplot[thick, teal]
{
    (abs(x) <= 1 ? \cval*(x >= 0 ? abs(x)^7 : -abs(x)^7) : \cval*(x >= 0 ? 1 : -1)) - 1
};

\addplot[dashed, gray, forget plot] coordinates {(-1,-1.95) (-1,-0.05)};
\addplot[dashed, gray, forget plot] coordinates {(1,-1.95) (1,-0.05)};

\end{axis}

\end{tikzpicture}

\caption{Left: \(f(x)=g_{p,c}(x)-x\). Right: \(f'(x)=g'_{p,c}(x)-1\), for \(p=4,6,8\) and \(c=0.8\).}
\label{fig:f-and-derivative}
\end{figure}
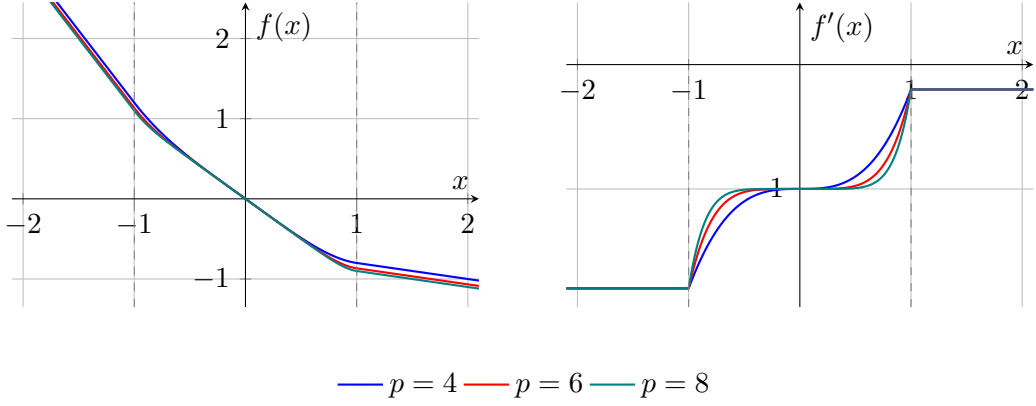
\section{A pathological example for composite problems}\label{section:2}

We now introduce the one-dimensional composite problem used in the derivation of the lower bound. Referring to Problem \eqref{eq:composite-problem}, we set $d=1$ and 
\begin{equation}
\label{def_f_reg}
    f(x):=g_{p,c}(x)-x, \qquad \varphi(x):=x,
\end{equation}
where \(g_{p,c}\colon\R\to\R\) is defined, for \(p>2\), \(c\in\left]0,1\right[\), by 
\[
g_{p,c}(x)
:=
\begin{cases}
\dfrac{c}{p}|x|^p, & |x|\le 1,\\[0.8em]
c\left(|x|-1+\dfrac1p\right), & |x|>1.
\end{cases}
\]
Moreover, $g_{p,c}$ is differentiable and its derivative is as follows
\begin{equation}
    \label{derivative}
    g'_{p,c}(x)
    =
    \begin{cases}
    c\,\operatorname{sign}(x)|x|^{p-1}, & |x|\le 1,\\[0.4em]
    c\,\operatorname{sign}(x), & |x|>1.
    \end{cases}
\end{equation}

\begin{proposition}[H\"older-smoothness]\label{prop:holder}
    The function \(f\) given in \eqref{def_f_reg}, with $p>2$ and $c\in \left]0,1\right[$, is convex and continuously differentiable. Moreover, \(f'\) is globally Lipschitz-continuous and globally bounded. Consequently, for every \(\nu\in\left]0,1\right]\), there exists \(L_\nu>0\) such that
\begin{equation*}
(\forall\,x,y\in \R)\quad
        \abs{f'(x)-f'(y)}
        \le
        L_\nu |x-y|^\nu
\end{equation*}
    meaning that \(f\in C^{1,\nu}(\R)\) for every \(\nu\in\left]0,1\right]\).
\end{proposition}

\begin{proof}
Let $p>2$ and $c\in \left]0,1\right[$.
The function $g'_{p,c}$ given in \eqref{derivative} is continuous on $\R$ and differentiable on $\R$
with the exception of $\pm 1$ and it holds
\begin{equation*}
    g''_{p,c}(x) =
\begin{cases}
    c(p-1)|x|^{p-2} &\text{if } x\in\left]-1,1\right[\\
    0 &\text{if}\ \R\setminus[-1,1].
\end{cases}
\end{equation*}
This shows that $\abs{g''_{p,c}(x)}\leq c(p-1)$, for every $x\in \R\setminus\{\pm 1\}$, and hence,
by the mean value theorem, that $g'_{p,c}$ is Lipschitz continuous with constant $c(p-1)$.
Consequently, recalling \eqref{def_f_reg}, it is clear that $f$ is also convex and Lipschitz-smooth, with constant $c(p-1)$ and, for every $x\in\R$,
\[
    f'(x)=g'_{p,c}(x)-1.
\]
Also, it follows from \eqref{derivative} that, for every $x\in\R$,
\[
    -c\le g'_{p,c}(x)\le c
\]
and hence
\[
    -(1+c)\le f'(x)\le -(1-c).
\]
Thus \(f'\) is globally bounded.
The case \(\nu=1\) follows directly from the global Lipschitz continuity of \(f'\). Let \(\nu\in\left]0,1\right[\), and set for the sake of brevity \(L=c(p-1)\). 
Let $x,y\in \R$.
If \(|x-y|\le1\), then
\[
    |f'(x)-f'(y)|
    \le
    L|x-y|
    \le
    L|x-y|^\nu.
\]
If \(|x-y|>1\), then, using boundedness of \(f'\),
\[
    |f'(x)-f'(y)|
    \le
    2c
    \le
    2c|x-y|^\nu.
\]
Therefore
\begin{equation*}
(\forall\,x,y\in \R)\quad
    \abs{f'(x)-f'(y)}
    \le
    \max\{L,2c\}|x-y|^\nu.
\end{equation*}
This proves the \(C^{1,\nu}\) property for every \(\nu\in\left]0,1\right]\).
\end{proof}

\begin{remark}
The derivative of the smooth part of $F$ does not vanish at the minimizer:
\[
    f'(0)=g'_{p,c}(0)-1=-1\neq0.
\]
This is the key feature of the construction. The point \(x_\star=0\) is a minimizer of the composite function, but the derivative of the smooth part \(f'(x_\star)\) is nonzero. Therefore an AdaGrad metric driven by \(|f'(x_n)|^2\) may keep growing linearly even when \(x_n\to x_\star\).
\end{remark}

\section{Analysis of the lower bound}

We now analyze the iterate \eqref{update} for the special objective components $f$ and $\varphi$ given in \eqref{def_f_reg}.
We first note that, directly from the definition the proximity operator, it follows that
\begin{equation*}
(\forall\,\gamma>0)(\forall\,v\in \R)\qquad \prox_{\gamma\varphi}(v)=v-\gamma.
\end{equation*}
\noindent
Let $n\in \N$.
If we denote with $\gamma_n = \eta/w_n$, recalling that $f'(x)=g'_{p,c}(x)-1$, we have
\[
\begin{aligned}
    x_{n+1}
    &=
    x_n-\gamma_n f'(x_n)-\gamma_n  \\
    &=
    x_n-\gamma_n(g'_{p,c}(x_n)-1)-\gamma_n \\
    &=
    x_n-\gamma_ng'_{p,c}(x_n).
\end{aligned}
\]
Thus, the affine part of \(f\) is exactly cancelled out by the proximal step associated with \(\varphi\). However, it remains in the AdaGrad denominator, which is built from \(f'(x_n)=g'_{p,c}(x_n)-1\).
By construction, \(|g'_{p,c}(x)|\le c<1\) for every \(x\in\R\). Hence, for every $x\in \R$, $f'(x)\leq 0$ and
\[
    1-c
    \le
    |f'(x)|
    \le
    1+c.
\]
Consequently,
\[
    \varepsilon+(1-c)\sqrt{n+1}
    \le
    w_n
    \le
    \varepsilon+(1+c)\sqrt{n+1}.
\]

\begin{lemma}[Convergence to the minimizer]
\label{lem:conv-zero}
Let $F=f+\varphi$, with $f$ and $\varphi$ defined in \eqref{def_f_reg}, with $p>2$ and $c\in \left]0,1\right[$.
Let $\eta$ be such that
\begin{equation*}
    0<\eta< \frac{1-c}{c(p-1)}
\end{equation*}
and let $(x_n)_{n\in \N}$ be generated by AdaGrad algorithm \eqref{update}, with $x_0\neq 0$.
Then \(x_n\to0\). Moreover, the sequence $(x_n)_{n\in \N}$ is nonincreasing and with constant sign.
\end{lemma}

\begin{proof}
Set \(L:=c(p-1)\). By Proposition~\ref{prop:holder}, \(g'_{p,c}\) is Lipschitz-continuous with constant \(L\), and \(g'_{p,c}(0)=0\). 
Let $n\in \N$. Since \(w_n\ge1-c\), we have
\begin{equation*}
    \gamma_n
    =
    \frac{\eta}{w_n}
    \le
    \frac{\eta}{1-c}
    <
    \frac1{c(p-1)}
    =
    \frac1L\ \Rightarrow\ \gamma_n L<1.
\end{equation*}
Suppose first that \(x_n>0\). Since \(0\le g'_{p,c}(x_n)\le Lx_n\), we get
\begin{equation*}
    x_{n+1}
    =
    x_n-\gamma_ng'_{p,c}(x_n)
    \ge
    x_n-\gamma_nLx_n
    >0.
\end{equation*}
Moreover, \(g'_{p,c}(x_n)\ge0\), and therefore \(x_{n+1}\le x_n\). Hence \(0< x_{n+1}\le x_n\).
Similarly, if \(x_n<0\), then \(g'_{p,c}(x_n)\le0\) and \(|g'_{p,c}(x_n)|\le L|x_n|\). Thus
\begin{equation*}
    x_{n+1}
    =
    x_n-\gamma_ng'_{p,c}(x_n) \leq     x_n-\gamma_nLx_n
    < 0,
\end{equation*}
while also \(x_{n+1}\ge x_n\). Hence \(x_n\le x_{n+1}<0\).
Therefore the sign of the iterates is preserved. In addition, \(r_n:=|x_n|\) is nonincreasing, and so \(r_n\) converges to some \(\bar r\ge0\). We prove that \(\bar r=0\). Suppose by contradiction that \(\bar r>0\). Since the sign of \(x_n\) is fixed and \(r_n\in[\bar r,r_0]\), where \(r_0:=|x_0|\), continuity of \(g'_{p,c}\) and the fact that \(g'_{p,c}(x)\neq0\) for \(x\neq0\) imply
\[
    a
    :=
    \inf_{\bar r\le |x|\le r_0}
    |g'_{p,c}(x)|
    >
    0.
\]
Thus \(|g'_{p,c}(x_n)|\ge a\) for every \(n\). Using
\[
    w_n\le \varepsilon+(1+c)\sqrt{n+1},
\]
we obtain
\[
    r_n-r_{n+1}
    =
    \gamma_n|g'_{p,c}(x_n)|
    \ge
    \frac{\eta a}{\varepsilon+(1+c)\sqrt{n+1}}.
\]
Summing this inequality gives a contradiction, since
\[
    \sum_{n=0}^{+\infty}
    \frac{1}{\varepsilon+(1+c)\sqrt{n+1}}
    =
    +\infty,
\]
but the left-hand side telescopes:
\[
    \sum_{n=0}^N(r_n-r_{n+1})
    =
    r_0-r_{N+1}
    \le r_0.
\]
Hence \(\bar r=0\), and therefore \(x_n\to0\).
\end{proof}

\begin{theorem}[Lower bound]
\label{lem:nonasymptotic-lower-bound}
Let $F=f+\varphi$, with $f$ and $\varphi$ defined in \eqref{def_f_reg}, with $p>2$ and $c\in \left]0,1\right[$.
Let $\eta$ be such that
\begin{equation*}
    0<\eta< \frac{1-c}{c(p-1)}
\end{equation*}
and let $(x_n)_{n\in \N}$ be generated by AdaGrad algorithm \eqref{update}, with $x_0\neq 0$. Then, there exists $\kappa>0$
such that
\begin{equation*}
(\forall\,n\in \N)\quad
    F(x_n)-F_\star
    \ge
    \kappa\,
    n^{-\frac{p}{2(p-2)}}.
\end{equation*}
\end{theorem}

\begin{proof}
Let \(r_n:=|x_n|\). By Lemma~\ref{lem:conv-zero}, \(r_n\downarrow0\), and the sign of \(x_n\) is constant. By oddness of \(g'_{p,c}\), the recursion becomes
\[
    r_{n+1}
    =
    r_n-\gamma_ng'_{p,c}(r_n).
\]
Set \(q:=p-2\). Since \(r_n\to0\), there exists \(\bar{n}\in\N\) such that \(0<r_n\le1\) for every \(n\ge \bar{n}\). Hence
\[
(\forall\,n\geq \bar{n})\quad    g'_{p,c}(r_n)=cr_n^{p-1}.
\]
Increasing \(\bar{n}\), if necessary, we may also assume that
\[
(\forall\,n\geq \bar{n})\quad
    \theta_n
    :=
    \gamma_n\frac{g'_{p,c}(r_n)}{r_n}
    \le \frac12.
\]
Indeed, \(g'_{p,c}(r_n)/r_n=cr_n^{p-2}\), \(r_n\to0\), and \(\gamma_n\le\eta/(1-c)\).
Now, let $n\in \N$ with $n\geq \bar{n}$.
Since \(0\le\theta_n\le1/2\), the mean value theorem applied to \(\theta\mapsto(1-\theta)^{-q}\) gives
\[
    (1-\theta_n)^{-q}-1
    \le
    q2^{q+1}\theta_n.
\]
Therefore, since \(r_{n+1}=r_n(1-\theta_n)\), we have
\[
\begin{aligned}
    r_{n+1}^{-q}-r_n^{-q}
    &=
    r_n^{-q}\left((1-\theta_n)^{-q}-1\right) \\
    &\le
    q2^{q+1}r_n^{-q}\theta_n \\
    &=
    q2^{q+1}\gamma_n
    \frac{g'_{p,c}(r_n)}{r_n^{q+1}}.
\end{aligned}
\]
Since \(q+1=p-1\) and \(g'_{p,c}(r_n)=cr_n^{p-1}\), we get
\[
    r_{n+1}^{-q}-r_n^{-q}
    \le
    q2^{q+1}c\,\gamma_n.
\]
Moreover,
\[
    w_n
    \ge
    \varepsilon+(1-c)\sqrt{n+1}
    \ge
    (1-c)\sqrt{n+1}.
\]
Thus
\[
    \gamma_n
    =
    \frac{\eta}{w_n}
    \le
    \frac{\eta}{(1-c)\sqrt{n+1}}.
\]
Consequently, for every \(n\ge \bar{n}\),
\[
    r_{n+1}^{-q}-r_n^{-q}
    \le
    \frac{q2^{q+1}c\eta}{(1-c)\sqrt{n+1}}.
\]
Summing from \(\bar{n}\) to \(n-1\), we obtain
\[
\begin{aligned}
    r_n^{-q}
    &\le
    r_{\bar{n}}^{-q}
    +
    \frac{q2^{q+1}c\eta}{1-c}
    \sum_{k=\bar{n}}^{n-1}\frac1{\sqrt{k+1}}  \\
    &\le
    r_{\bar{n}}^{-q}
    +
    \frac{2q2^{q+1}c\eta}{1-c}\sqrt n.
\end{aligned}
\]
In the end there exists \(A>0\) such that, for every integer $n\geq \bar{n}$,
\[
    r_n^{-q}\le A\sqrt n,
\]
and hence
\[
    r_n
    \ge
    A^{-1/q}n^{-\frac1{2q}}
    =
    A^{-\frac1{p-2}}n^{-\frac1{2(p-2)}}.
\]
Finally, using the above lower bound on \(r_n\), we obtain
\[
    F(x_n)-F_\star
    =
    g_{p,c}(x_n)
    =
    g_{p,c}(r_n) = \frac{c}{p}r_n^p
    \ge
    C_1\,
    n^{-\frac{p}{2(p-2)}},
\]
for every integer $n\geq \bar{n}$, where
\[
    C_1
    :=
    \frac{c}{p}A^{-\frac{p}{p-2}}>0.
\]
It remains to absorb the finite number of iterations before $\bar{n}$. Since \(x_0\neq0\), Lemma~\ref{lem:conv-zero} shows that the sign of \(x_n\) is preserved, and in particular \(x_n\neq0\) for every \(n\). Therefore \(F(x_n)-F_\star>0\) for every \(n\). Define
\[
    C_2
    :=
    \min_{1\le n<\bar{n}}
    n^{\frac{p}{2(p-2)}}
    \bigl(F(x_n)-F_\star\bigr) > 0,
\]
with the convention that \(C_2=+\infty\) if \(\bar{n}\le1\). Setting
\[
    \kappa:=\min\{C_1,C_2\}
\]
gives, 
\begin{equation*}
    F(x_n)-F_\star
    \ge
    \kappa\,
    n^{-\frac{p}{2(p-2)}},
\end{equation*}
for every \(n\ge1\), and the statement follows.
\end{proof}

\begin{corollary}[Failure of AdaGrad for every H\"older-smooth class]
\label{thm:main}
Fix $\nu\in \left]0,1\right]$ and $\alpha$
such that
\begin{equation*}
\frac 1 2 < \alpha<\frac{1+\nu}{2}.
\end{equation*}
Then there exists $f \in C^{1,\nu}(\R)$ and $\varphi\colon \R\to \R$ affine such that,
if $(x_n)_{n\in \N}$ is generated by AdaGrad algorithm~\eqref{update} with $x_0\neq 0$ and sufficiently small $\eta>0$,
then there exists $\kappa>0$ such that
\begin{equation*}
(\forall\,n\in \N)\qquad 
F(x_n)-F_\star
    \geq
    \kappa n^{-\alpha}.
\end{equation*} 
Consequently, AdaGrad can have a slow rate which is as close as desired to $ \mathcal O(1/\sqrt{n})$ on convex composite problems with \(C^{1,\nu}\) smooth part, for any $\nu\in \left]0,1\right]$, and in particular it does not, in general, achieve the H\"older-smooth rate
\begin{equation*}
    \mathcal O\left(n^{-\frac{1+\nu}{2}}\right),
\end{equation*}
which was proved for unconstrained \(C^{1,\nu}\) smooth problems in \citep{orabona2023normalized}.
\end{corollary}

\begin{proof}
Set $p = 4\alpha /(2\alpha-1)>2$. Then
\begin{equation*}
\frac{p}{2(p-2)} = \alpha \in \left]\frac1 2, \frac{1+\nu}{2}\right[.
\end{equation*}
Choose $f$ and $\varphi$ as in Theorem~\ref{lem:nonasymptotic-lower-bound} (with $p$ given above and $c\in \left]0,1\right[$).
Then, by Proposition~\ref{prop:holder}, \(f\in C^{1,\nu}(\R)\) and the statement follows from Theorem~\ref{lem:nonasymptotic-lower-bound}.
\end{proof}
\begin{remark}[Averaged iterate]
The same lower bound holds for the averaged iterate
\[
    \bar x_n:=\frac1n\sum_{k=1}^n x_k.
\]
Indeed, the sign of \(x_n\) is preserved and \(|x_n|\) is nonincreasing. Hence
\[
    |\bar x_n|\ge |x_n|.
\]
Since \(F=g_{p,c}\) is even and nondecreasing as a function of \(|x|\), we have
\[
    F(\bar x_n)-F_\star
    \ge
    F(x_n)-F_\star.
\]
Thus, for every integer $n\geq 1$,
\[
    F(\bar x_n)-F_\star
    \ge
    \kappa n^{-\alpha}.
\]
\end{remark}

\section{Discussion: what should be accumulated?}\label{section:4}

The lower bound above is not due to stochastic noise, nonconvexity, or lack of smoothness of \(f\). Rather, it arises from using the wrong stationarity measure in a composite problem. In the construction of the pathological example, since the derivative of the smooth component does not vanish, we force the accumulator to grow linearly. Consequently, the effective stepsize satisfies
\[
\frac{\eta}{w_n}
= \mathcal{O}\left( \frac{1}{\sqrt{n}} \right),
\]
even though the iterates approach the minimizer. This is the critical issue: as an example, in the Lipschitz-smooth setting, the textbook proximal gradient method uses a constant stepsize, whereas here the decay of the effective stepsize is unavoidable and prevents the standard convergence rates from being recovered.\\

\noindent
A more suitable quantity for the accumulation mechanism in the composite setting is the gradient mapping, introduced in this context by \citet{antonakopoulos2025adaptive,wang2026universal}. For a fixed \(\gamma>0\), define the gradient mapping
\[
    \gradmap_\gamma(x)
    :=
    \frac1\gamma
    \left(
        x-\prox_{\gamma\varphi}
        \bigl(x-\gamma f'(x)\bigr)
    \right).
\]
This quantity satisfies
\[
    \gradmap_\gamma(x_\star)=0
\]
whenever \(x_\star\) is a minimizer of the composite objective $f+\varphi$ and the gradient mapping is expected to vanish as the iterates converge to a solution. Thus an adaptive mechanism that accumulates
\[
    \sum_{k=0}^n |\gradmap_{\eta/w_k}(x_k)|^2
\]
would not keep adding a nonzero constant merely because \(f'(x_\star)\neq 0\). \\

\noindent
An alternative way to avoid the same pathology is to accumulate gradient differences rather than gradient magnitudes. This is the idea behind AdaGrad-Diff in \citep{bojovic2026adagrad}. In the scalar case, the difference-based metric takes the form
\[
w_{n}^{\mathrm{diff}}
=
\varepsilon+
\left(
\sum_{k=0}^n |f'(x_k)-f'(x_{k-1})|^2
\right)^{1/2},
\qquad
f'(x_{-1})=0.
\]
The crucial distinction is that, even if the gradient magnitudes \(|f'(x_n)|\) do not vanish, the successive differences still decay at the scale needed for appropriate convergence rate. In the Lipschitz-smooth case, these differences are square-summable, so the corresponding stepsize remains of constant order. In the H\"older-smooth case \(f\in C^{1,\nu}\), it is not difficult to show, following the ideas in \citep{bojovic2026adagrad}, that the growth is instead
\[
{w_n^{\mathrm{diff}}}
=
\mathcal O\left(n^{\frac{1-\nu}{2}}\right).
\]
This is precisely aligned with the stepsize decay associated with the H\"older-smooth setting for unconstrained optimization problems considered in \citep{orabona2023normalized}.
Therefore the metric is not forced to grow linearly simply because the limiting gradient of the smooth part is nonzero. \\

\noindent
In this sense, both gradient-mapping accumulation and AdaGrad-Diff-type accumulation target quantities that are better aligned with composite optimality. The example above shows that this distinction is not merely cosmetic: for standard gradient accumulation, the mismatch can rule out the standard \(\mathcal{O}(n^{-(1+\nu)/2})\) rate even in a deterministic one-dimensional convex H\"older-smooth composite problem.

\section{Conclusion}

We have provided an example showing that Adagrad does not adapt to the level of H\"older-smoothness of the objective function in a composite optimization setting.
 The critique relies on the fact that the gradient of the smooth part may not vanish at a  minimizer for a composite problem. 
Our example suggests that, in the composite setting, adaptive mechanisms should accumulate quantities consistent with the composite optimality condition. Two natural candidates are the gradient mapping and successive gradient differences. The gradient mapping vanishes at composite minimizers, while successive gradient differences do not accumulate a nonzero limiting gradient. Both therefore avoid the failure mechanism identified here: the linear growth of the gradient accumulator even as the iterates converge to a composite minimizer with $f'(x_\star)\neq 0$.
\bibliographystyle{plainnat}
\bibliography{ref}

@article{bojovic2026adagrad,
  title={Ada{G}rad-{D}iff: A {N}ew {V}ersion of the {A}daptive {G}radient {A}lgorithm},
  author={Bojovic, Matia and Salzo, Saverio and Pontil, Massimiliano},
  journal={arXiv preprint arXiv:2602.13112},
  year={2026}
}

@article{JMLR:v12:duchi11a,
  author  = {John Duchi and Elad Hazan and Yoram Singer},
  title   = {Adaptive {S}ubgradient {M}ethods for {O}nline {L}earning and {S}tochastic {O}ptimization},
  journal = {Journal of Machine Learning Research},
  year    = {2011},
  volume  = {12},
  number  = {61},
  pages   = {2121--2159},
}

@article{wang2026universal,
  title={Universal {A}daptive {P}roximal {G}radient {M}ethods via {G}radient {M}apping {A}ccumulation},
  author={Wang, Zimeng and Yurtsever, Alp},
  journal={arXiv preprint arXiv:2605.05944},
  year={2026}
}

@article{orabona2023normalized,
  title={Normalized gradients for all},
  author={Orabona, Francesco},
  journal={arXiv preprint arXiv:2308.05621},
  year={2023}
}

@article{levy2018online,
  title={Online adaptive methods, universality and acceleration},
  author={Levy, Kfir Y and Yurtsever, Alp and Cevher, Volkan},
  journal={Advances in neural information processing systems},
  volume={31},
  year={2018}
}

@article{antonakopoulos2025adaptive,
  title={Adaptive bilevel optimization},
  author={Antonakopoulos, Kimon and Sabach, Shoham and Viano, Luca and Hong, Mingyi and Cevher, Volkan},
  journal={ACM/IMS Journal of Data Science},
  volume={2},
  number={2},
  pages={1--29},
  year={2025},
  publisher={ACM New York, NY}
}

@article{nesterov2015universal,
  title={Universal gradient methods for convex optimization problems},
  author={Nesterov, Yu},
  journal={Mathematical Programming},
  volume={152},
  number={1},
  pages={381--404},
  year={2015},
  publisher={Springer}
}

\end{document}